\documentclass{amsart}
\usepackage{amsfonts,amssymb,amscd,amsmath,enumerate,verbatim,calc}

\newcommand{\CM}{Cohen-Macaulay}

\newcommand{\wrt}{with respect to}

\newcommand{\n}{\mathfrak{n} }
\newcommand{\m}{\mathfrak{m} }

\newcommand{\ve}{\varepsilon}

\newcommand{\pd}{\operatorname{projdim}}

\newcommand{\socle}{\operatorname{Soc}}
\newcommand{\Syz}{\operatorname{Syz}}

\newcommand{\Hom}{\operatorname{Hom}}
\newcommand{\Ext}{\operatorname{Ext}}

\theoremstyle{plain}

\newtheorem{theorem}{Theorem}[section]
\newtheorem{corollary}[theorem]{Corollary}
\newtheorem{lemma}[theorem]{Lemma}
\newtheorem{proposition}[theorem]{Proposition}

\theoremstyle{definition}

\newtheorem{s}[theorem]{}
\newtheorem{remarks}[theorem]{Remarks}
\newtheorem{remark}[theorem]{Remark}
\newtheorem{example}[theorem]{Example}

\begin{document}
\title[Dual Hilbert-Samuel polynomial]{The dual Hilbert-Samuel function of a Maximal Cohen-Macaulay module}
\author{Tony~J.~Puthenpurakal}
\address{Department of Mathematics, IIT Bombay, Powai, Mumbai 400 076, India}
\email{tputhen@math.iitb.ac.in}
\author{Fahed~Zulfeqarr}
\address{Department of Mathematics, IIT Bombay, Powai, Mumbai 400 076, India}
\date{today}
\subjclass{Primary    13D45 ; Secondary 13A30}
\keywords{multiplicity,  reduction, Hilbert-Samuel polynomial }

\begin{abstract}
Let $R$ be a Cohen-Macaulay local ring with a canonical module $\omega_R$. Let $I$ be an $\m$-primary ideal of $R$ and  $M$, a  maximal Cohen-Macaulay $R$-module. We call the function $n\longmapsto \ell \left( \Hom_R(M,{\omega_R}/{I^{n+1} \omega_R})\right)$ the {\it dual Hilbert-Samuel function of $M$ with respect to $I$}. By a result of Theodorescu  this function is a polynomial function. We study its first two normalized coefficients.
\end{abstract}
\date{\today}
\maketitle

\section*{Introduction}
\noindent Let  $(R,\m )$ be a Cohen-Macaulay local ring of dimension $d$ with a canonical module $\omega_R$. If the ring $R$ is clear from context we write $\omega_R$ as $\omega$. Let $I$ be an $\m$-primary ideal and $M$, a finitely generated $R$-module. \\
\indent The Hilbert-Samuel function of $M$ with respect to $I$ is the function
\[
n\longmapsto \ell\left( M\otimes \frac{R}{I^{n+1}}\right)\quad \text{for all}\  n\geqslant 0.
\]
It is well-known (see \cite[4.6.2]{BH}) that for all large values of $n$, this function
is given by a polynomial of degree equal to $\dim M$. This polynomial is called the {\it Hilbert-Samuel polynomial of $M$} with respect to $I$. Its normalized leading coefficient $e_0^I(M)$ is called the {\it multiplicity of $M$ with respect to $I$}.\\
\indent  When $M$ is Cohen-Macaulay $R$-module of dimension $r$, we define the {\it dual Hilbert-Samuel function with respect to $I$} by
\[
n \longmapsto \ell\left( \Ext^{d-r}_R(M,\frac{\omega}{I^{n+1}\omega})\right)\quad\mbox{for all}\;n\geqslant 0,
\]
We call the power-series
\[
D(M,t) =  \sum_{n\geq 0}\ell\left( \Ext^{d-r}_R(M,\frac{\omega}{I^{n+1}\omega})\right)z^n
\]
as the \emph{dual Hilbert-Samuel series} of $M$ \wrt \ $I$.
We concentrate on the case when $\dim M=\dim R$, i.e., when $M$ is a maximal \CM \  (MCM) $R$-module.

 For $i\geqslant 0$ we call the function
\[
n\longmapsto \ell\left( \Ext_R^i(M,\frac{\omega}{I^{n+1}\omega})\right)\quad\mbox{for all}\;\; n\geqslant 0
\]
as the {\it $i$-th dual Hilbert-Samuel function of $M$ with respect to} $I$.  We study these functions for $i=0,\;1$ and when $M$ is MCM. Using Theodorescu's result \cite[4]{Theo}, it follows that this function is given by a polynomial in $n$ for all large values of $n$.
         We denote this polynomial by $\varepsilon^i_M(I,t)$ and call it the {\it $i$-th dual Hilbert-Samuel polynomial of $M$ with respect to $I$}. Furthermore we have
\begin{itemize}
\item [$($\rm 1$)$] for $i=0$, the dual Hilbert-Samuel polynomial $\varepsilon^0_M(I,t)$
is of degree $d$.
\item [$($\rm 2$)$] for  each $i>0$,  the $i$-th dual Hilbert-Samuel polynomial $\varepsilon^i_M(I,t)$ is of degree at most $ d-1$.
\end{itemize}
Therefore the polynomial $\varepsilon^0_M(I,t)$ can be written  in the following form
\[
\varepsilon^0_M(I,t)=\sum_{i=0}^d(-1)^ic_i^I(M) \binom{t + d - i}{d-i}.
\]
We call the integers \;$c_0^I(M),c_1^I(M),...,c_d^I(M)$ \;the {\it dual Hilbert coefficients of} $M$ with respect to $I$. In Proposition \ref{hiltcoef} we show that
\[
c_0^I(M)=e_0^I(M).
\]
This  also shows  that if $J$ is a  reduction of $I$ \wrt \ $M$ then $c_0^I(M) = c_0^J(M)$.

One might also try to plausibly study the "Dual Hilbert function" of a MCM module $M$, i.e., the function
\[
\delta_I(M,n) =  \ell \left( \Hom_R(M, I^n\omega/I^{n+1}\omega) \right) \quad \text{for} \ n \geq 0.
\]
However we feel that it is not as interesting as the Dual Hilbert-Samuel function for the following reasons: Assume $R$ is Gorenstein and $I = \m$.
\begin{itemize}
  \item The dual Hilbert function is
  \[
  \delta_\m(M,n) =  \mu(M)\ell\left(\m^n/\m^{n+1} \right).
  \]
  Here $\mu(M) $ denotes the number of minimal generators of $M$.
  Thus this function tells very little information on $M$.
  \item
  If  $J$ is a minimal reduction of $\m$ \wrt \   $M$  then it can be easily shown that
  the normalized leading coefficient of the polynomial function $\delta_J(M,n)$ is  $e_0^J(M) = e_0^\m(M)$. Using the above calculation we get
  that the  normalized leading coefficient of $\delta_\m(M,n)$ is  $\mu(M)e_0^\m(A)$. Since $M$ is MCM we get that
  $e_0^\m(M) = \mu(M)e_0^\m(A)$ if and only if $M$ \emph{is free.} Thus for interesting cases we do  have that
  leading coefficients of $\delta_J(M,n)$ and $\delta_\m(M,n)$ are different.
\end{itemize}

A crucial property of the dual Hilbert-Samuel function is that it behaves well modulo  a sufficiently general element in $I \setminus \m I$.  More specifically in Proposition  \ref{behevsuper} we show the following:
Let $x \in I\setminus \m I$ be sufficiently general.
Set $S=R/xR$, $J=I/{(x)}$ and $N=M/{xM}$. We prove

$$ c_i^J(N) = c_i^I(N) \quad \text{for} \ i = 0,\ldots,d-1. $$

In Theorem \ref{polydeg} we show that if $M$ is not a free MCM $R$-module then the function
\begin{equation*}
n\longmapsto \ell \left( \Ext^1_R(M,\frac{\omega}{{\m}^{n+1}\omega})\right)\quad \mbox{for }\;n\geqslant 0 \tag{*}
\end{equation*}
is a polynomial type of degree $d-1$. Example 2.6 shows that $\deg \varepsilon^1_M(I,t)$ can be  $-\infty$ in general.
A consequence of (*) is that if $M$ is a non-free MCM $R$-module then
\[
\mu(M)e_1(\omega_A) > c_1(M) + c_1(\Syz^R_1(M)).
\]

In case the ring $R$ is Gorenstein and its associated ring $G_I(R)$ is Cohen-Macaulay, we prove (see Theorem \ref{relbetcoeffs}) that if $r$ is a reduction number of $I$ then we have
\[
c_1^I(M) \geqslant r \cdot e_0^I(M)- \sum_{j=0}^{d}\sum_{n=0}^{r-1}\binom{\;d\;}{j}\ell\left(\Ext^j_R(M,\frac{R}{{I}^{n-j+1}}) \right).
\]
In view of the above result we study the function
\[
\Phi^I(M)=\sum_{j=0}^{d}\sum_{n=0}^{r-1}\binom{\;d\;}{j}\ell\left(\Ext^j_R(M,\frac{R}{{I}^{n+1-j}}) \right).
\]
If $x\in I \setminus  \m I$ is an $R$-superficial element then we prove
\[
\Phi^I(M)\geqslant \Phi^{I/(x)}(M/xM).
\]

The dual Hilbert-Samuel function behaves well \wrt \  the basic properties just as the usual Hilbert-Samuel function. However
for computational reasons it compares poorly \wrt \ the usual Hilbert function.
In section  6 we compute the dual Hilbert-Samuel series of an Ulrich module $M$ over a Gorenstein local ring $R$ such that $G_\m(R)$ is also Gorenstein. This computation is modulo a result from \cite{TJPW}.

We now describe in brief the contents of this paper. In section 1 we introduce notation
and discuss a few preliminary facts that we need. In section 2 we prove that $c_0^I(M) = e_0^I(M)$.  In section 3 we discuss the behavior of the functions
$\ve^0, \ve^1$ modulo \ a general element in $I \setminus \m I$. In section 4 we prove that the polynomial $\ve^1_M(\m,t)$ is either zero or a
polynomial of degree $d-1$. In section 5 we discuss a lower bound on $c_1^I(M)$. Finally in the last section we  compute the dual Hilbert-Samuel series of an Ulrich module $M$ over a Gorenstein local ring $R$ such that $G_\m(R)$ is also Gorenstein.
\section{Notation and Preliminaries}
\noindent Throughout this paper unless otherwise stated we assume that $R$ is a Cohen-Macaulay (CM) local ring of dimension $d$ with maximal ideal $\m$ and that  residue field $R/{\m}$ is infinite. We also assume that $R$ has a canonical module $\omega$. Let $I$ be an $\m$-primary ideal and $M$, a finitely generated  $R$-module.

\s {\bf Notation :}\label{notations1}
 We denote the Rees Algebra of ideal $I$ by $\mathcal R(I)$, the associated graded ring with respected to $I$ by $G_I(R)$ and their associated graded modules by $\mathcal R(I,M)$ and $G_I(M)$ respectively, i.e.,
\begin{eqnarray*}
\mathcal R(I)=\bigoplus_{n\geqslant 0} I^n,&&\quad G_I(R)=\bigoplus_{n\geqslant 0}\frac{I^n}{I^{n+1}}\\
\mbox{and}\quad \mathcal R(I,M)=\bigoplus_{n\geqslant 0} I^nM,&& G_I(M)=\bigoplus_{n\geqslant 0}\frac{I^nM}{I^{n+1}M}
\end{eqnarray*}
Clearly $\mathcal R(I,M)$ is a finitely generated $\mathcal R(I)$-module and $G_I(M)$ is a finitely generated $G_I(R)$-module.\\
\indent Since $I$ is $\m$-primary ideal in $R$ then by \cite[4.5.13]{BH}, $l(I)=\dim R=d$. Using \cite[3(c)]{Theo}, since $\omega$ is faithful $R$-module, we get $l_{\omega}(I)=l(I).$ Thus $l_{\omega}(I)=d$.

\begin{remarks}\label{introfrems}
From \cite[4]{Theo}, it follows that the function
\[
 n \longmapsto  \ell\left( \Ext_R^i(M,\frac{\omega}{I^{n+1} \omega})\right)\quad\mbox{for each}\;i\geqslant 0
 \]
 is given by a polynomial, say $\varepsilon^i_M(I,n)$, in $n$ for all $n\gg 0$ and
\begin{equation*}
\deg \varepsilon^i_M(I,t) \leqslant \max \left\lbrace \dim \Ext_R^i(M,\omega),l_{\omega}(I)-1\right\rbrace . \tag{*}
\end{equation*}
Moreover equality holds if $\dim \Ext_R^i(M,\omega) \geqslant l(I)=d$. If $M$ is an $MCM$ $R$-module then using (*) we get that the function
\begin{itemize}
\item [$($\rm a$)$] $n\longmapsto \ell\left( \Hom_R(M,{\omega}/{I^{n+1}\omega})\right)$ is a polynomial type of degree $d$.
\item [$($\rm b$)$] $n\longmapsto \ell\left( \Ext^1_R(M,{\omega}/{I^{n+1}\omega})\right)$ is a polynomial type of degree at most $d-1$.
\end{itemize}
\end{remarks}

\section{Relation with Hilbert Coefficients}
\noindent In this section we establish an equality between $e^I_0(M)$ and the dual Hilbert coefficient $c_0^I(M)$ of $M$.
\s \label{notaofext}
\noindent{\bf Notation:} We set
\begin{eqnarray*}
\mathcal F \equiv \{M^{\dagger}_n\}_{n\geqslant 0} &\mbox{where} & M^{\dagger}_n = \Hom_R(M,I^n\omega)\quad\mbox{for all}\;\;n\geqslant 1\\
E^i_I(M) = \bigoplus_{n\geqslant 0} \Ext^i_R(M,I^n \omega) &\mbox{and}&  D^i_I(M) =\bigoplus_{n\geqslant 0} \Ext^i_R(M,\omega/{I^n \omega}).
\end{eqnarray*}

\begin{remarks}\label{hilcof}$~$
\begin{itemize}
\item [$($\rm a$)$] It can be easily seen that $\mathcal F$ is an stable $I$-filtration. Therefore by \cite[11.4]{Atiyah}, the function
\[
n\longmapsto \ell\left( M^{\dagger}/M_n^{\dagger}\right)
\]
is given by a polynomial of degree $\dim M^{\dagger}=d$, for all $n\gg 0$. Further the normalized leading coefficient of this function is $e_0^I(M^{\dagger})$.
\item [$($\rm b$)$]\label{relofmul}
Notice that $M$ and $\Hom_R(M,\omega)$ are both $MCM$ $R$-modules, It is well known and can be easily verified that
\[
e_0^I(M^{\dagger})=e_0^I(M).
\]
\item[$($\rm c$)$]
Let $E = \bigoplus_{n\geq 0}E_n$ be a finitely generated $\mathcal{R}(I)$-module such that $\ell_R(E_n) < \infty$ for all $n \geq 0$. It is well
known that the function $n \mapsto \ell(E_n)$ is polynomial of degree $\leq d -1$.
\end{itemize}
\end{remarks}

\indent The following lemma is crucial.

\begin{lemma}\label{hiltfgext1}
Let $M$ be a finitely generated $R$-module. Then for each $i\geqslant 0$, $E^i_I(M)$  is finitely generated graded ${\mathcal R}(I)$-module.
\end{lemma}
\begin{proof} Consider a minimal free resolution of $M$
\[
\mathbb F_{\bullet} \quad :\quad ...\rightarrow F_2\rightarrow F_1\rightarrow F_0\rightarrow  0.
\]
 Applying the functor $\Hom_R(-,\bigoplus_{n\geqslant 0}I^n\omega)$ to this sequence,  we get the complex $\Hom_R(\mathbb F_{\bullet},\bigoplus_{n\geqslant 0}I^n\omega)$ of finitely generated graded ${\mathcal R}(I)$-modules. We now have
\begin{eqnarray*}
\bigoplus_{n\geqslant 0}\Ext_R^i(M,I^n\omega)&\simeq & \Ext_R^i(M,\bigoplus_{n\geqslant 0}I^n\omega)\\
&=& \mathcal H^i\left( \Hom_R(\mathbb F_{\bullet},\bigoplus_{n\geqslant 0}I^n\omega)\right).
\end{eqnarray*}
Hence for each $i\geqslant 0$, $\bigoplus_{n\geqslant 0}\Ext_R^i(M,I^n\omega)$ is finitely generated graded ${\mathcal R}(I)$-module.
\end{proof}

\begin{lemma}\label{hiltfgext}
Let $M$ be an MCM $R$-module. Then the following hold :
\begin{itemize}
\item [$($\rm a$)$] For $i\geqslant 1$, $D^i_I(M) \cong E^{i+1}_I(M)$. So $D^i_I(M)$ is a finitely generated graded ${\mathcal R}(I)$-module.
\item [$($\rm b$)$] For all $i\geqslant 1,$ the function $n\mapsto \ell\left( \Ext_R^i(M,I^n\omega)\right) $ is a polynomial function of degree $\leqslant d-1$.
\end{itemize}
\end{lemma}
\begin{proof} {\bf (a):} Consider the exact sequence
\[
0\rightarrow I^n\omega \rightarrow \omega \rightarrow  \frac{\omega}{I^n\omega} \rightarrow 0.
\]
Since  $M$ is $MCM$, $\Ext_R^i(M,\omega)=0$ for all  $i\geqslant 1$. So applying the functor $\Hom_R(M,-)$ to the above exact sequence we get
\begin{eqnarray*}
\Ext_R^i(M,\omega/{I^n\omega}) &\simeq & \Ext_R^{i+1}(M,I^n\omega)\quad\mbox{for all}\;i\geqslant 1.
\end{eqnarray*}
From Lemma \ref{hiltfgext1} it follows that $D^i_I(M)$ is finitely generated graded ${\mathcal R}(I)$-module  and  $ D^i_I(M)\cong E^{i+1}_I(M) $  for all $i\geqslant 1$.

\noindent {\bf (b):}  Since $I$ is $\m$-primary and by using part (a), we get that for all $i\geqslant 2$, $\ell \left(\Ext_R^i(M,I^n\omega) \right) < \infty$ for all $n$. Therefore by \ref{hilcof}(c) the function \\ $n\longmapsto \ell\left( \Ext_R^i(M,I^n\omega)\right)$ is a polynomial function of degree at most $l(I)-1=d-1$. For $i=1$, consider the exact sequence
\[
0 \rightarrow I^n\omega \rightarrow \omega \rightarrow  \frac{\omega}{I^n\omega} \rightarrow 0.
\]
Applying the functor $\Hom_R(M,-)$ we get
\[
 0 \rightarrow \frac{M^{\dagger}}{M^{\dagger}_n}\rightarrow \Hom_R(M,\frac{\omega}{I^n\omega}) \rightarrow \Ext_R^1(M,I^n\omega) \rightarrow 0.
\]
It follows that $\ell\left( \Ext_R^1(M,I^n\omega)\right)<\infty $ for all $n\geqslant 1$. Using \ref{hilcof}(c) we get that for all $n\gg 0$, the function $ n\longmapsto \ell\left( \Ext_R^1(M,I^n\omega)\right)$ is given a polynomial of degree at most $d-1$.
\end{proof}

\indent Another useful observation is recorded in the following proposition.

\begin{proposition}\label{hiltcoef}
The function $n\longmapsto \ell \left( \Hom_R(M,{\omega_R}/{I^{n+1} \omega_R})\right) $ is of polynomial type and the dual Hilbert coefficient $c_0^I(M)$ is same as the multiplicity of $M$ with respect to $I$.
\end{proposition}
\begin{proof}
Consider the exact sequence
\[
0 \rightarrow I^n\omega \rightarrow \omega \rightarrow  \frac{\omega}{I^n\omega} \rightarrow 0.
\]
Applying the functor $\Hom_R(M,-)$ we get
\[
 0 \rightarrow \frac{M^{\dagger}}{M^{\dagger}_n}\rightarrow \Hom_R(M,\frac{\omega}{I^n\omega}) \rightarrow \Ext_R^1(M,I^n\omega) \rightarrow 0.
\]
Therefore it follows that
\[
 \ell\left( \Hom_R(M,\frac{\omega}{I^n\omega})\right) = \ell\left(\frac{M^{\dagger}}{M^{\dagger}_n} \right) +\ell\left(\Ext_R^1(M,I^n\omega)  \right).
\]
By Lemma \ref{hiltfgext}(b), we have that the function $ n\longmapsto \ell\left( \Ext_R^1(M,I^n\omega)\right)$ is given a polynomial function of degree at most $d-1$. Also by Remark \ref{hilcof}(a), the function $n\longmapsto \ell\left( M^{\dagger}/{M^{\dagger}_n}\right)$ is a polynomial type of degree $d$. So the function $ n\longmapsto \ell \left( \Hom_R(M,{\omega_R}/{I^{n+1} \omega_R})\right)$ is given a polynomial $\varepsilon^0_M(I,t)$ and
the leading coefficient of this polynomial is same as that of $n\longmapsto \ell\left( M^{\dagger}/{M^{\dagger}_n}\right)$. Therefore  $c_I^0(M)=e_0^I(M^{\dagger})$. Remark \ref{hilcof}(b) gives that $c_0^I(M)=e_0^I(M)$.\end{proof}

We end this section by computing the dual Hilbert-Samuel function \wrt \ a parameter ideal $J$. The reasons for this computations are the following
\begin{enumerate}
  \item If $J$ is a minimal reduction of $I$ \wrt \ $\m$ then $J$ is a parameter ideal. We also have
  \[
  c^I_0(M) = e_0^I(M) = e_0^J(M) = c^J_0(M).
  \]
  Here the first and the third equality is due to Proposition \ref{hiltcoef} while the second equality is a basic property of  reductions.
  \item The computation will also prove that
  \[
  \Ext^{1}_{R}\left(M, \frac{\omega}{J^{n+1}\omega}\right) = 0 \quad \text{for all} \ n \geq 0.
  \]
  This assertion proves that we cannot hope that Theorem 4.2 to work for all $\m$-primary ideals.
  \item We need to compute the Dual Hilbert-Samuel function in this essentially simplest case of an $\m$-primary ideal.
\end{enumerate}
\begin{example}
Let $M$ be a MCM module. Let $J = (x_1,\ldots,x_d)$ be a parameter ideal. Then
\begin{enumerate}[\rm (a)]
  \item
$ \displaystyle{  \ve_M^0(J,n) =  \ell\left( \Hom_R(M, \frac{\omega}{J^{n+1}\omega} \right) =  \ell(M/JM)\binom{n+d}{d}    }$
  \item
$\displaystyle{  \ve_M^1(J,n) =  \ell\left( \Ext^1_R(M, \frac{\omega}{J^{n+1}\omega} \right) =  0 \quad \text{for all} \ n \geq 0.   } $
\end{enumerate}
\end{example}
\begin{proof}
Let $\delta(n) = \mu(J^n) = \binom{n+d-1}{d-1}$. Since $x_1,\ldots, x_d$ is a $R$-regular  sequence, it is also $\omega$-regular. So we get
\[
\frac{J^n\omega }{J^{n+1} \omega} = \left(\frac{\omega}{J\omega}\right)^{\delta(n)}.
\]
Set $S = R/J$. Notice $\omega_S = \omega/J\omega$. We will also use the fact that for $i \geq 0$,
\begin{align*}
\ell\left(\Ext^{i}_{R}(M, \omega/J\omega) \right) &= \ell\left(\Ext^{i}_{S}(M/JM, \omega_S)\right), \quad \text{cf. \cite[\S 18, Lemma 2]{Mat} }\\
 &= \begin{cases} \ell(M/JM), &\text{if $i = 0$;} \\ 0, &\text{if $i \geq 1$.}                       \end{cases} \tag{*}
\end{align*}
We prove assertions (a), (b) by induction on $n$.
For $n = 0 $ the result follows from (*).
For $n \geq 0$ we use the exact sequence
\[
0 \longrightarrow \frac{J^n\omega}{J^{n+1}\omega} \longrightarrow \frac{\omega}{J^{n+1}\omega} \longrightarrow \frac{\omega}{J^{n}\omega} \longrightarrow 0.
\]
We now apply the functor $\Hom_R(M,-)$. The corresponding long exact sequence and (*) gives the result.
\end{proof}

\section{Behavior with respect to a Superficial elements}
\noindent In this section we discuss the behavior of $\varepsilon^0_M(\omega,t)$ with respect to a superficial element. To ensure good behavior, see \ref{behevsuper},   it is important to choose
superficial elements with some care, see \ref{supassump}.

Let us first recall some basic definitions.
\s An element $x\in I$ is called $M$-${\it superficial}$ with respect to $I$ if there exists an integer
$c\ge 0$ such that
$$(I^{n+1}M:_M x)\cap I^cM=I^nM\quad \mbox{for all}\; n\ge c.$$
Also recall that if $E=\bigoplus_{n\geqslant 0}E_n$ is finitely generated graded ${\mathcal R}(I)$-module then an element $xt\in {\mathcal R}(I)_1$ is called $E$-${\it filter\; regular}$ if
\begin{eqnarray*}
(0:_E x)_n &=& 0\quad \mbox{for all}\;n\gg 0. \\
\mbox{Equivalently }\quad \alpha_x \;:\; E_n  &\rightarrow & E_{n+1}\quad \mbox{are injective for}\;n\gg 0,
\end{eqnarray*}
where $\alpha_x$ is the  map induced by multiplication by $x$ (i.e., $\alpha_x(u)=xu$). Notice that if the residue field $R/{\m}$ is infinite then superficial and filter regular element exists (see \cite[18.3.10]{BS}).

\s \label{supassump}{\bf Choosing a superficial element:}
Since by Lemma \ref{hiltfgext},  $D^i_I(M)$ are finitely generated ${\mathcal R}(I)$-modules for $i\geqslant 1$, we can choose an element $x\in I\setminus I^2$ such that
\begin{itemize}
\item [$($\rm a$)$] $x$ is $I$-superficial element with respect to $R\oplus\omega$.
\item [$($\rm b$)$] $xt\in {\mathcal R}(I)_1$ is $D^1_I(M)\oplus D^2_I(M)$-filter regular.
\end{itemize}
Clearly we have $x$ is $R\oplus M$-regular. Set $S=R/xR$, $J=I/{(x)}$ and $N=M/{xM}$. Notice that $\omega_S=\omega/{x\omega},$ the canonical module of $S$.

\begin{proposition}\label{behevsuper}
 Let $x\in I\setminus I^2$ be an element  satisfying  \ref{supassump}. Then
\begin{itemize}
\item [$($\rm a$)$] $\varepsilon^0_M(I,n) - \varepsilon^0_M(I,n-1)= \varepsilon^0_{N}(J,n)$\quad for $n \gg 0$.
\item [$($\rm b$)$]  $\varepsilon^1_M(I,n) - \varepsilon^1_M(I,n-1)= \varepsilon^1_{N}(J,n)$\quad for $n \gg 0$.
\end{itemize}
\end{proposition}
\begin{proof} Consider the exact sequence 	
\begin{eqnarray*}
0\rightarrow \frac{I^{n+1}\omega: x}{I^n\omega}\rightarrow \frac{\omega}{I^n \omega} &\xrightarrow{\alpha_x^n} &\frac{\omega}{I^{n+1} \omega} \rightarrow \frac{\omega}{(x,I^{n+1}) \omega} \rightarrow 0,\\
\mbox{such that}\quad {\alpha_x^n}(u+I^n\omega)&=& xu+I^{n+1}\omega,\quad\forall \;u\in \omega.
\end{eqnarray*}
As $x$ is $\omega$-superficial and $\omega$-regular, we have $I^{n+1}\omega: x=I^n\omega$ for $n\gg 0$. Set $J=I/(x)$. For $n\gg 0$, one can thus write above exact sequence as follows
\begin{equation*}
0\rightarrow \frac{\omega}{I^n \omega} \xrightarrow{\alpha_x^n} \frac{\omega}{I^{n+1} \omega} \rightarrow \frac{\omega_S}{J^{n+1}\omega_S} \rightarrow 0,
\end{equation*}
where $\omega_S=\omega_R/{x\omega_R}$, the canonical module of $S$. Now applying the functor $\Hom_R(M,-)$ we get a long exact sequence
\begin{eqnarray*}
0\rightarrow \Hom_R(M,\frac{\omega}{I^n \omega}) \rightarrow & \Hom_R(M,\frac{\omega}{I^{n+1} \omega}) &\rightarrow \Hom_R(M,\frac{\omega_{S}}{J^{n+1}\omega_{S}})\\
\rightarrow \Ext^1_R(M,\frac{\omega}{I^n \omega}) \rightarrow & \Ext^1_R(M,\frac{\omega}{I^{n+1} \omega}) &\rightarrow \cdots
\end{eqnarray*}
By our hypotheses on $x$, the maps
\begin{eqnarray*}
\Ext^1_R(M,\alpha_x^n)\colon \Ext^1_R(M,\frac{\omega}{I^n \omega}) &\rightarrow & \Ext^1_R(M,\frac{\omega}{I^{n+1} \omega})\\
\Ext^2_R(M,\alpha_x^n)\colon \Ext^2_R(M,\frac{\omega}{I^n \omega}) &\rightarrow & \Ext^2_R(M,\frac{\omega}{I^{n+1} \omega})
\end{eqnarray*}
are injective for all $n\gg 0$. Therefore for all $n\gg 0$, we get the exact sequences
\begin{eqnarray*}
0\rightarrow  \Hom_R(M,\frac{\omega}{I^n \omega}) \rightarrow  &\Hom_R(M,\frac{\omega}{I^{n+1} \omega}) &\rightarrow
 \Hom_R(M,\frac{\omega_{S}}{J^{n+1}\omega_{S}}) \rightarrow  0\\
0\rightarrow \Ext^1_R(M,\frac{\omega}{I^n \omega}) \rightarrow   &\Ext^1_R(M,\frac{\omega}{I^{n+1} \omega})& \rightarrow \Ext^1_R(M,\frac{\omega_{S}}{J^{n+1}\omega_{S}}) \rightarrow 0.
\end{eqnarray*}
Using \cite[Lemma 18.2]{Mat}, we get
\begin{eqnarray*}
 \Ext^i_R(M,\frac{\omega_{S}}{J^{n+1}\omega_{S}}) &\cong & \Ext^i_{S}(N,\frac{\omega_{S}}{J^{n+1}\omega_{S}})\quad\mbox{for all }\;n\geqslant 0.
\end{eqnarray*}
The result follows.
\end{proof}

The following is a very useful consequence to above proposition.

\begin{corollary}\label{behrem}
Let $x\in I\setminus I^2$ be an element  satisfying  \ref{supassump}. Then
\[
c_i^J(N)=c_i^I(M)\quad \mbox{for}\;\;i=0,1,2,...,d-1.
\]
\end{corollary}
\begin{proof}
It follows easily from  Proposition \ref{behevsuper}(a) and  the binomial identity
\[
\binom{n + d - i-1}{d-i-1} = \binom{n + d - i}{d-i}-\binom{n + d - i-1}{d-i}.
\]
\end{proof}

\section{The function $n\longmapsto \ell\left(\Ext^1_R(M,\omega/{{\m}^{n+1}\omega})\right)$}
\noindent In this section our goal is to show that for a non-free $MCM$ $R$-module $M$ of dimension $d\geqslant 1$, the degree of polynomial $\varepsilon_M^1(\m,t)$ with respect to maximal ideal $\m$ is exactly $d-1$. This need not be true for all $\m$-primary ideals in general (see Example 2.6.\\
 Recall the $i$-th betti number $\beta_i$ of $M$ (see \cite[1.3.1]{BH}) is given by
\[
\beta_i=\dim_k \Ext^i_R(M,k),\quad\mbox{where}\;\;k=R/{\m}.
\]
Note that $\beta_0=\mu(M)$, the minimal number of generators of $M$.

\indent The following is the dual version  of Proposition 17  in \cite{TJP2}.

\begin{proposition}\label{polydeg}
Let \;$\dim M=d\geqslant 1$. Set $L=\Syz^1_R(M)$, the first syzygy of $M$. Then
\begin{itemize}
\item [$($\rm a$)$] We have
$$\sum_{n\geqslant 0}\ell\left( \Ext^1_R(M,\omega/{I^{n+1}\omega})\right)t^n=\frac{f^I_M(t)-\mu(M)h^I_{\omega}(t)+f^I_L(t)}{(1-t)^{d+1}}.$$
\item [$($\rm b$)$] For $n\gg 0$ the function $n\longmapsto \ell \left( \Ext^1_R(M,\omega/{I^{n+1}\omega})\right)$ is given by a polynomial $\varepsilon^1_M(I,t)$ of the form
\[
\varepsilon^1_M(I,t)=\left( \mu(M)\cdot e_1^I(\omega)-c_1^I(M)-c_1^I(L))\right) \frac{t^{d-1}}{{(d-1)!}}+ \mbox{lower degree terms}.
\]
\item [$($\rm c$)$] $c_1^I(M)+c_1^I(L))\leqslant \mu(M)\cdot e_1^I(\omega).$
\end{itemize}
\end{proposition}
\begin{proof}
{\bf (a):} Consider the exact sequence
\begin{equation*}
0\longrightarrow L\longrightarrow R^{\mu(M)} \longrightarrow M \longrightarrow 0.\tag{i}
\end{equation*}
Applying $\Hom_R(-,\omega/{I^{n+1}\omega})$ to this, we get
\begin{eqnarray*}
0 \rightarrow \Hom_R(M,\frac{\omega}{I^{n+1}\omega}) &\rightarrow & \Hom_R( R^{\mu(M)},\frac{\omega}{I^{n+1}\omega}) \rightarrow \Hom_R(L,\frac{\omega}{I^{n+1}\omega}) \\
\rightarrow \Ext^1_R(M,\frac{\omega}{I^{n+1}\omega}) &\rightarrow & 0.
\end{eqnarray*}
This induces that
\begin{eqnarray*}
\sum_{n\geqslant 0}\ell\left( \Ext^1_R(M,\omega/{I^{n+1}\omega})\right)t^n &=& \frac{f^I_M(t)}{(1-t)^{d+1}}- \frac{\mu(M)\cdot h^I_{\omega}(t)}{(1-t)^{d+1}}+\frac{f^I_L(t)}{(1-t)^{d+1}}\\
&=& \frac{f^I_M(t)-\mu(M)\cdot h^I_{\omega}(t)+f^I_L(t)}{(1-t)^{d+1}}
\end{eqnarray*}
\noindent{\bf (b):} Now we set $p(t)=f^I_M(t)-\mu(M)\cdot h^I_{\omega}(t)+f^I_L(t)$. From Proposition \ref{hiltcoef},
it follows that
\[
p(1)=e_0^I(M)-\mu(M)\cdot e_0^I(\omega)+e_0^I(L)=0,\quad \mbox{by (i)}.
\]
 We thus write $p(t)=(1-t)s(t)$. This gives that
\[
\sum_{n\geqslant 0}\ell\left( \Ext^1_R(M,\omega/{I^{n+1}\omega})\right)t^n=\frac{s(t)}{{(1-t)}^d}
\]
This shows that $\deg \varepsilon^1_M(I,t)\leqslant d-1$. We also have
\[
s(1)=-p'(1)=\mu(M)\cdot e_1^I(\omega)-c_1^I(M)-c_1^I(L).\\
\]
Hence
\[
\varepsilon^1_M(I,t)= \left( \mu(M)\cdot e_1^I(\omega)-c_1^I(M)-c_1^I(L)\right)\frac{t^{d-1}}{(d-1)!}+ \mbox{lower terms in}\; t.
\]
\noindent{\bf (c):} This follows easily from (b).
\end{proof}

In the following theorem we take $I=\m$. We prove that if $M$ is non-free $MCM$ $R$-module then $\deg \varepsilon^1_M(\m,t)=d-1$.

\begin{theorem}\label{condoffree}
Let $R$ be Cohen-Macaulay local of dimension $d\geqslant 1$. The following conditions are equivalent
\begin{itemize}
\item [$($\rm a$)$] $M$ is free.
\item [$($\rm b$)$] $\varepsilon^1_M(\m,t)=0.$
\item [$($\rm c$)$] $\deg \varepsilon^1_M(\m,t) < d-1$.
\item [$($\rm d$)$] $c_1(M)+c_1(\Syz^1_R(M))=\mu(M)\cdot e_1(\omega)$.
\end{itemize}
\end{theorem}

Example 2.6 shows that this result need not true for all $m$-primary ideals.

\begin{proof} [Proof of the theorem \ref{condoffree}]
The implications (a) $\Rightarrow$ (b) $\Rightarrow$  (c)  are clear. The assertion
(c) $\Rightarrow$ (d) follows from Proposition 4.1(b).

\noindent {\bf (d) $\Rightarrow$ (a):} We apply induction on $d=\dim M$. If $d=1$ then $\varepsilon^1_M(\m,t)=0.$ So we get
\[
\Ext^1_R(M,\frac{\omega}{\m^{n+1}\omega})=0 \quad\mbox{for all}\;n\gg 0.
\]
The exact sequence
\[
0\longrightarrow \frac{{\m}^n\omega}{{\m}^{n+1}\omega} \longrightarrow \frac{\omega}{{\m}^{n+1}\omega} \longrightarrow \frac{\omega}{{\m}^n \omega} \longrightarrow 0
\]
gives rise to a long exact sequence, for $n\gg 0$,
\begin{eqnarray*}
0\rightarrow \Hom_R(M,\frac{{\m}^n\omega}{\m^{n+1}\omega}) &\rightarrow & \Hom_R(M,\frac{\omega}{\m^{n+1}\omega}) \rightarrow \Hom_R(M,\frac{\omega}{\m ^n\omega}) \\
\rightarrow \Ext_R^1(M,\frac{{\m}^n\omega}{\m^{n+1}\omega}) &\rightarrow & 0.
\end{eqnarray*}
Therefore  we get a relation
\[
\mu(M)\cdot e_0(\omega) - \left[ e_0(M)\cdot n-c_1(M)\right]+\left[ e_0(M)\cdot (n-1)-c_1(M)\right]- \beta_1 \cdot e_0(\omega) =  0.
\]
\begin{equation*}\\[-1mm]
\mbox{So,}\quad e_0(M) = \mu(M)\cdot e_0(\omega)-\beta_1 \cdot e_0(\omega) = \mu(M)\cdot e_0(R)-\beta_1 \cdot e_0(R).\tag{i}
\end{equation*}
Set $L_1=\Syz_R^1(M)$ and $L_2=\Syz_R^1(L_1)$. As $M$ is $MCM$ so are $L_1$ and $L_2$.
The exact sequence
\[
0\rightarrow L_1\rightarrow R^{\mu(M)} \rightarrow M \rightarrow 0
\]
yields
\begin{equation*}
e_0(L_1) = \mu(M)\cdot e_0(R) - e_0(M).\tag{ii}
\end{equation*}
Similarly we get
\begin{equation*}
 e_0(L_2) = \beta_1\cdot e_0(R) - e_0(L_1).\tag{iii}
\end{equation*}
Therefore we have
\begin{eqnarray*}
e_0(L_2) &=& \beta_1\cdot e_0(R)-[ \mu(M)\cdot e_0(R) -e_0(M)] \quad \left(\mbox{ by (ii)}\right) \\
&=&  [ \beta_1\cdot e_0(R)- \mu(M)\cdot e_0(R)]  +e_0(M)\\
&=&-e_0(M) +e_0(M)=0 \quad \left( \mbox{ by (i)}\right).
\end{eqnarray*}
So $L_2 = 0$  and hence $\pd_R(M) <\infty$. Using Auslander-Buchsbaum's formula we get $M$ is free. \\
\indent When $d>1$. We choose \;$x_1,x_2,...,x_d\in \m\setminus \m^2$ such that
\begin{itemize}
\item [$($\rm 1$)$] $x_i$ is $R_{i-1}\oplus \omega_{i-1}$-superficial element,
\item [$($\rm 2$)$] $x_i\in {\mathcal R}(\m)_1$ is \;$E_1(M_{i-1})\oplus E_2(M_{i-1})$-filter regular,
\end{itemize}
where $R_i=R/{(x_1,...,x_i)}$ and $M_i=M/{(x_1,...,x_i)M}$ for $i=1,2,..,d$. Set
\[
J=(x_1,x_2,...,x_{d-1}), \quad S=R/J \quad \mbox{and}\quad N=M/{JM}.
\]
Since $\;x_1,x_2,...,x_{d-1}\;$ is regular on $R$ and $M$, $\omega$ and $L$ are $MCM$, we get
\begin{eqnarray*}
\Syz^1_{S}(N) \cong \Syz_R^1(M)/{J\Syz_R^1(M)}, &\omega_S=\omega_R/{J\omega_R}& \mbox{and} \quad \mu(M)=\mu(N).
\end{eqnarray*}
By \cite[11(1)]{TJP2}, we have
\[
e_j(\omega_R)=e_j(\omega_S)\quad\mbox{and}\quad e_j(M)=e_j(N) \quad\mbox{for}\;j=0,1.
\]
Since $M$ and $\Syz^R_1(M)$ are $MCM$ $R$-modules, by Corollary \ref{behrem}, we have
\[
c_1(M)=c_1(N)\quad \mbox{and}\quad c_1(\Syz_R^1(M)) = c_1(\Syz^1_{S}(N)).
\]
We now consider
\begin{eqnarray*}
c_1(N)+c_1(\Syz^1_{S}(N)) &=& c_1(M)+c_1(\Syz_R^1(M))\\
&=& \mu(M)\cdot e_1(\omega_R),\quad\mbox{by our hypothesis}\\
&=& \mu(M)\cdot e_1(\omega_S).
\end{eqnarray*}
Note that $\dim {N}=1$. So, by induction hypothesis, $N$ is free $S$-module. So $\Ext_{S}^1(N,k)=0$. But
$\Ext_R^{d+1}(M,k)\cong \Ext_{S}^1(N,k)$. This gives $\Ext_R^{d+1}(M,k)=0$ and so $\beta_{d+1}=0$. Therefore $\pd_R(M)<\infty$. Since $M$ is $MCM$, using  Auslander-Buchsbaum's formula we get $M$ is free.
\end{proof}

\section{The case when ring is Gorenstein and its associated graded ring is Cohen-Macaulay}
\noindent In this section we assume that $(R,\m)$ is Gorenstein local ring of dimension $d$  such that its associated graded ring $G_I(R)$ is $CM$. Let $r$ be a reduction number of $I$.  Then we prove
\[
c_1^I(M) \geqslant r \cdot e^I_0(M)- \sum_{n=0}^{r-1}\sum_{j=0}^{d}\binom{\;d\;}{j}\ell\left(\Ext^j_R(M,\frac{R}{{I}^{n+1-j}}) \right).
\]
The above inequality motivates us to investigate the function
\[
\Phi^I(M)=\sum_{j=0}^{d}\sum_{n=0}^{r-1}\binom{\;d\;}{j}\ell\left(\Ext^j_R(M,\frac{R}{{I}^{n+1-j}}) \right).
\]
This we do. We  show that if $x\in I\setminus I^2$ is an $R$-superficial element then we have
\[
\Phi^I(M) \geqslant \Phi^{I/(x)}(M/xM).
\]

\s \label{0dimensionalcase}
{\bf $0$-Dimensional Case:}
 We first deal with $0$-dimensional Gorenstein ring $S$ with unique maximal ideal $\n$.
Let $J$ be a $\n$-primary ideal and $N$ be a finitely generated $S$-module. Since $S$ is $0$-dimensional, there exists a positive integer $r\in \mathbb N$ such that $J^{r+1}= 0$ but $J^r\neq 0$. We have
\begin{eqnarray*}
\sum_{n=0}^{\infty} \ell\left( \Hom_{S}(N, \frac{S}{J^{n+1}S})\right) t^n &=& \frac{f_N^J(t)}{1-t}\\
\sum_{n=0}^{r-1} \ell\left( \Hom_{S}(N, \frac{S}{J^{n+1}S})\right) t^n +\frac{\ell\left(\Hom_S(N,S) \right) }{1-t}\cdot t^r &=& \frac{f_N^J(t)}{1-t}
\end{eqnarray*}
By Matlis duality, \;$\ell\left(\Hom_S(N,S) \right)=\ell(N)$. So $c_0^J(N)=e_0^J(N)$. Thus
\begin{eqnarray*}
\sum_{n=0}^{r-1} \ell\left( \Hom_{S}(N, \frac{S}{J^{n+1}S})\right) t^n + \frac{e_0^J(N)}{1-t}\cdot t^r &=& \frac{f_N^J(t)}{1-t}\\
\mbox{so,}\quad f_N^J(t) = (1-t)\cdot \sum_{n=0}^{r-1} \ell\left( \Hom_{S}(N, \frac{S}{J^{n+1}S})\right) t^n &+& e_0^J(N)\cdot t^r.
\end{eqnarray*}
Set $\alpha_n=\ell\left(\Hom_{S}(N, S/{J^{n+1}S})\right)$. Therefore we get
\begin{equation*}
f_N^J(t) = (1-t)\cdot \sum_{n=0}^{r-1}\alpha_n t^n + e_0^J(N)\cdot t^r.\tag{*}
\end{equation*}
Using (*) we compute $c_i^J(N)$ for $i = 0,1$. Note that
\begin{eqnarray*}
c_0^J(N) &=& f_N^J(1) = e_0^J(N)\\
\mbox{and}\quad c_1^J(N) &=& \left[ \frac{d}{dt}f_N^J(t)\right]_{t=1} =r\cdot e_0^J(N) -\sum_{n=0}^{r-1}\alpha_n .
\end{eqnarray*}
\s\label{generalcase} {\bf General Case:}
Let $(R,\m)$ be Gorenstein local ring of dimension $d\geqslant 1$ with $\m$-primary ideal $I$ such that the associated graded ring $G_I(R)$ is $CM$.  We choose $\underline x=x_1,x_2,...,x_d\in I\setminus I^2$ be a sequence such that
\begin{itemize}
\item [$($\rm 1$)$] $x_i$ is $I$-superficial element with respect to $R_{i-1}$,
\item [$($\rm 2$)$] $x_i\in {\mathcal R}(I)_1$ is \;$D_1(M_{i-1})$-filter regular,
\end{itemize}
where $R_i=R/{(x_1,...,x_i)}$, $I_i=I/{(x_1,...,x_i)}$, $M_i=M/{(x_1,...,x_i)M}$ for \;$i=1,2,..,d$ and $D_1(M_i)=\bigoplus_{n\geqslant 0}\Ext_{R_i}^1(M_i,R_i/{I_i^n})$.

In the following proposition we set $S=R/{(\underline x)}$, $\n=\m/{(\underline x)}$, $J=I/{(\underline x)}$  and $N=M/{(\underline x)M}$.

\begin{proposition}\label{supbehforGor}
$($with hypotheses as in \ref{generalcase}$)$ We have
\[c_1^I(M)\geqslant c_1^J(N).\]
\end{proposition}
\begin{proof} By Corollary \ref{behrem}, we get
\begin{equation*}
c_1^I(M)=c_1^{I_1}(M_1)= \cdots = c_1^{I_{d-1}}(M_{d-1}).\tag{i}
\end{equation*}
So in view of (i), it is enough to show that
\[
c_1^{I_{d-1}}(M_{d-1}) \geqslant c_1^J(N).
\]
Thus we may assume that the dimension of ring is $1$, i.e., $d=1$.
Since  $G_{I}(R)$ is Cohen-Macaulay then $R$-superficial element $x \in I\setminus{I^2}$ implies that $x^* \in G_I(R)_1$ is $G_{I}(R)$-regular. Set $S=R/{(x)}$, $N=M/{xM}$ and $J=I/(x)$. Thus we get a short exact sequence of the form
\begin{eqnarray*}
0\rightarrow \frac{R}{I^n} \xrightarrow{\alpha_n^x} \frac{R}{I^{n+1}}\rightarrow \frac{S}{{J}^{n+1}} \rightarrow 0.
\end{eqnarray*}
This gives the following long exact sequence
\begin{eqnarray*}
0\rightarrow \Hom_{R}(M,\frac{R}{{I}^n}) \rightarrow  \Hom_{R}(M,\frac{R}{{I}^{n+1}}) \rightarrow  \Hom_{S}(N,\frac{S}{{J}^{n+1}}) \rightarrow  K_n \rightarrow 0,
\end{eqnarray*}
where $K_n=\ker\left( \Ext^1_{R}(M,{R}/{{I}^n})\rightarrow \Ext^1_{R}(M,{R}/{{I}^{n+1}})\right)$. By construction $K_n=0\;$ for all $n\gg 0$. Therefore we have
\begin{eqnarray*}
\frac{f_{R}^{M}(t)}{(1-t)} +\sum_{n\geqslant 0}\ell(K_n)t^n &=& \frac{f_{S}^{N}(t)}{(1-t)}\\
\Rightarrow\quad f_{S}^{N}(t) &=&  f_{R}^{M}(t)+ (1-t)\cdot \sum_{n\geqslant 0}\ell(K_n)t^n
\end{eqnarray*}
It follows that
\begin{eqnarray*}
c_1^J(N) &=&  c_1^{I}(M)-\sum_{n\geqslant 0}\ell(K_n)\\
\Rightarrow\quad c_1^J(N) &\leqslant & c_1^{I}(M).
\end{eqnarray*}
This completes the proof.
 \end{proof}

Before stating our main theorem  we first prove the following lemma.

\begin{lemma}\label{extrelation}
Let $x\in I\setminus I^2$ be an $R$-superficial element. Set $R^\prime =R/{(x)}$,  $I^\prime = I/(x)$ and $M^\prime=M/{xM}$ and . Then for $j\geqslant 0$ and $n \geq 1$ we have
\begin{enumerate}[\rm (i)]
\item
$ \Ext^j_{R\prime}(M^\prime, R^\prime/I^\prime) = \Ext^j_{R}(M, R/I)$
\item
$ \displaystyle{\ell\left(\Ext^j_{R^\prime}(M^\prime,\frac{R^\prime}{{I^\prime}^{n+1}})\right) \leqslant  \ell\left( \Ext^j_{R}(M,\frac{R}{{I}^{n+1}})\right) +
 \ell\left( \Ext^{j+1}_{R}(M,\frac{R}{{I}^n})\right).}$
 \end{enumerate}
\end{lemma}
\begin{proof}
(i) The first assertion follows from \cite[\S 18, Lemma 2]{Mat}.

(ii) Clearly $x\in I\setminus I^2$ is $G_I(R)$-regular. Thus for each $n \geqslant 1$, we get an exact sequence of the form
\[
0\rightarrow \frac{R}{{I}^n} \xrightarrow{\;\alpha_n^x\;} \frac{R}{{I}^{n+1}}\rightarrow \frac{R^\prime}{{I^\prime}^{n+1}} \rightarrow 0.
\]
Applying $\Hom_R(M,-)$, we get  following long exact sequence
\begin{eqnarray*}
\cdots\rightarrow \Ext^j_{R}(M,\frac{R}{{I}^{n+1}}) \rightarrow  \Ext^j_{R}(M,\frac{R^\prime}{{I^\prime}^{n+1}}) \rightarrow \Ext^{j+1}_{R}(M,\frac{R}{{I}^{n}}) \rightarrow \cdots
\end{eqnarray*}
 From \cite[\S 18, Lemma 2]{Mat} we get
 \[
 \Ext^j_{R}(M,\frac{R^\prime}{{I^\prime}^{n+1}})  = \Ext^j_{R^\prime}(M^\prime,\frac{R^\prime}{{I^\prime}^{n+1}})
 \]
The result follows.
\end{proof}

We now make a convention that will be used throughout the section.
\s\label{con}
{\bf Convention:} For ideal $I$, we set $I^j=R$ for $j\leqslant 0$.

\indent A useful inequality is recorded in the following corollary.

\begin{corollary}\label{useful}
$($with hypotheses as in \ref{generalcase}$)$ For all $n\geqslant 1$, we have
\[
\ell\left( \Hom_{S}(N, \frac{S}{J^n})\right)  \leqslant \sum_{j=0}^{d}\binom{d}{j} \cdot \ell\left( \Ext_R^{j}(M,\frac{R}{{I}^{n-j}})\right)
\]
\end{corollary}
\begin{proof}
By repeated use of Lemma \ref{extrelation} we notice that
\begin{eqnarray*}
\ell\left( \Hom_{S}(N, \frac{S}{J})\right)  &=& \ell\left( \Hom_{R}(M, \frac{R}{I})\right) \\
\ell\left( \Hom_{S}(N, \frac{S}{J^2})\right)  &\leqslant & \ell\left( \Hom_{R}(M, \frac{R}{I^2})\right)  + \binom{d}{1} \cdot \ell\left( \Ext^1_R(M,\frac{R}{I})\right)
\end{eqnarray*}
Our convention gives that
\begin{eqnarray*}
\ell\left( \Hom_{S}(N, \frac{S}{J^2})\right) &\leqslant & \sum_{j=0}^{d}\binom{d}{j} \cdot \ell\left( \Ext_R^{j}(M,\frac{R}{{I}^{2-j}})\right).
\end{eqnarray*}
Similarly one can check that for all $n\leqslant d$,
\begin{eqnarray*}
\ell\left( \Hom_{S}(N, \frac{S}{J^n})\right)  &\leqslant & \sum_{j=0}^{d}\binom{d}{j} \cdot \ell\left( \Ext_R^{j}(M,\frac{R}{{I}^{n-j}})\right).
\end{eqnarray*}
\end{proof}

A relation between the dual Hilbert coefficients $c_0^I(M)$ and $c_1^I(M)$ is established in the following theorem.

\begin{theorem}\label{relbetcoeffs}
$($with hypotheses as in \ref{generalcase}$)$ Let $r$ be the reduction number of $I$. Then
\[
c_1^I(M) \geqslant r \cdot e^I_0(M)- \sum_{n=0}^{r-1}\sum_{j=0}^{d}\binom{\;d\;}{j}\ell\left(\Ext^j_R(M,\frac{R}{{I}^{n+1-j}}) \right)
\]
\end{theorem}
\begin{proof} By Proposition \ref{supbehforGor}, we have $c_1^I(M)\geqslant c_1^J(N)$. Since $e_0^J(N)=e_0^I(M)$ then we get
\begin{equation*}
c_1^I(M) \geqslant r\cdot e_0^I(M) -\sum_{n=0}^{r-1}\ell\left(\Hom_{S}(N, \frac{S}{J^{n+1}})\right).
\end{equation*}
By Corollary \ref{useful}, we have
\begin{eqnarray*}
\ell\left( \Hom_{S}(N, \frac{S}{J^{n+1}})\right)  &\leqslant & \sum_{j=0}^{d}\binom{d}{j} \cdot \ell\left( \Ext_R^{j}(M,\frac{R}{{I}^{n+1-j}})\right)
\end{eqnarray*}
Therefore
\begin{eqnarray*}
\sum_{n=0}^{r-1}\ell\left(\Hom_{S}(N, \frac{S}{J^{n+1}})\right) &\leqslant & \sum_{n=0}^{r-1}\sum_{j=0}^{d}\binom{\;d\;}{j}\ell\left(\Ext^j_R(M,\frac{R}{{I}^{n+1-j}}) \right)
\end{eqnarray*}
The result follows.
\end{proof}
\s The previous result motivates us to study
\[
\Phi^I(M)=\sum_{j=0}^{d}\sum_{n=0}^{r-1}\binom{\;d\;}{j}\ell\left(\Ext^j_R(M,\frac{R}{{I}^{n+1-j}}) \right),
\]
for a $MCM$ $R$-module $M$ of dimension $d$ with respect to ideal $I$.

\begin{remarks}\label{forcomp}
We can simplify the expression for $\Phi^I(M)$.
\begin{itemize}
\item [$($\rm a$)$] Notice that
\[
\Ext_R^{j}(M,\frac{R}{{I}^{n+1-j}})=0\quad\mbox{for all}\;n< j.
\]
Therefore we get
\[
 \Phi^I(M)  = \sum_{j=0}^{d}\sum_{n=j}^{r-1}\binom{d}{j} \cdot \ell\left( \Ext_R^{j}(M,\frac{R}{{I}^{n+1-j}})\right)
\]
\item [$($\rm b$)$] If $r-1\geqslant d$, then we have
\[
\phi^I(M) =  \sum_{j=0}^{r-1}\sum_{n=j}^{r-1}\binom{d}{j} \cdot \ell\left( \Ext_R^{j}(M,\frac{R}{{I}^{n+1-j}})\right)
\]
\end{itemize}
\end{remarks}

\begin{proposition}\label{functinequlity}
Let $x\in I\setminus I^2$ be an $R$-superficial element. Then we have
\[
\Phi^I(M)\geqslant \Phi^{I/(x)}(M/xM)
\]
\end{proposition}
\begin{proof} For brevity set $R_1=R/(x)$, $I_1=I/(x)$ and $M_1=M/{xM}$. Since $G_I(R)$ is Cohen-Macaulay and $x\in I\setminus I^2$ an $R$-superficial then we have an exact sequence of the form
\[
0\rightarrow \frac{R}{{I}^m} \xrightarrow{\;\alpha_m^x\;} \frac{R}{{I}^{m+1}}\rightarrow \frac{R_1}{{I}_1^{m+1}} \rightarrow 0.
\]
This gives a long exact sequence of the form
\begin{equation*}
\cdots \rightarrow \Ext^j_{R}(M,\frac{R}{{I}^{m+1}})\rightarrow  \Ext^j_{R_1}(M_1,\frac{R_1}{{I}_1^{m+1}}) \rightarrow\Ext^{j+1}_R(M,\frac{R}{{I}^{m}}) \rightarrow \cdots.\tag{*}
\end{equation*}
By definition of $\Phi$, as $\dim M_1=d-1$, we have
\begin{eqnarray*}
\Phi^{I_1}(M_1) &=& \sum_{j=0}^{d-1}\sum_{n=0}^{r-1}\binom{d-1}{j}\ell\left(\Ext^j_{R_1}(M_1,\frac{R_1}{{I}_1^{n+1-j}}) \right)
\end{eqnarray*}
From (*), it follows that
\begin{eqnarray*}
\Phi^{I_1}(M_1) &\leqslant & \sum_{j=0}^{d-1}\sum_{n=0}^{r-1}\binom{d-1}{j} \ell\left(\Ext^{j}_R(M,\frac{R}{{I}^{n+1-j}})\right) + \\
&& \;\; \sum_{j=0}^{d-1}\sum_{n=0}^{r-1}\binom{d-1}{j}\ell\left(\Ext^{j+1}_R(M,\frac{R}{{I}^{n-j}})\right)
\\
&=& \sum_{j=0}^{d-1}\sum_{n=0}^{r-1}\binom{d-1}{j} \ell\left(\Ext^{j}_R(M,\frac{R}{{I}^{n+1-j}})\right) + \\
&& \;\; \sum_{j=1}^{d}\sum_{n=0}^{r-1}\binom{d-1}{j-1}\ell\left(\Ext^{j}_R(M,\frac{R}{{I}^{n+1-j}})\right).
\end{eqnarray*}
Notice that
\[
 \binom{d-1}{j}=0\quad\mbox{for}\ j=d \ \text{and} \ \quad \binom{d-1}{j-1}=0\quad\mbox{for}\;\;j=0.
\]
Therefore using a well-known Binomial identity we get
\[
 \Phi^{I_1}(M_1) \leqslant \sum_{j=0}^{d}\sum_{n=0}^{r-1}\binom{d}{j}\ell\left(\Ext^j_{R}(M,\frac{R}{{I}^{n+1-j}}) \right).
\]

Hence it follows that $\Phi^{I_1}(M_1) \leqslant \Phi^I(M)$.
\end{proof}

The following is an immediate consequence.

\begin{corollary}
Let the situation be as in \ref{generalcase}. Then we have
\[
\Phi(M)\geqslant \Phi(M_1)\geqslant \Phi(M_2)\geqslant \cdots \geqslant \Phi(M_d)=\Phi(N).
\]
\end{corollary}

\begin{remark}\label{exprin0}
      Let $S$ be $0$-dimensional Gorenstein local ring, and assume that $J^{r+1}= 0$ but $J^r\neq 0$. Then
\[
\Phi^{J}(N)=\sum_{n=0}^{r-1}\ell\left(\Hom_{S}(N, \frac{S}{J^{n+1}})\right)
\]
\end{remark}

\begin{example}\label{exam1}
Let $R=\mathbf{Q}[x,y]/{(f)}$ be a hypersurface ring of dimension $d=1$, where $f=x^2+xy+y^2$ a homogeneous polynomial of degree $2$. Let $\m=\left\langle x,y \right\rangle R$ be a maximal ideal of $R$. Consider an $R$-module
$$M= \left\langle \left(\begin{array}{cc} x \\  -y  \end{array} \right),\left(\begin{array}{cc} x+y \\  x  \end{array} \right) \right\rangle \;\subseteq\; R^2 .$$
Note that $M$ is  an $MCM$ $R$-module. Clearly reduction number $r$ of $\m$ is $1$. Therefore we get
\begin{eqnarray*}
\Phi^{\m}(M) &=& \sum_{j=0}^{d}\sum_{n=0}^{r-1}\binom{\;d\;}{j}\ell\left(\Ext^j_R(M,\frac{R}{{\m}^{n+1-j}}) \right)\\
&=& \ell\left(\Hom_R(M,\frac{R}{{\m}}) \right)\\
&=& \mu(M)=2.\\
\Rightarrow\quad r\cdot e_0(M)  - \Phi^{\m}(M) & =& 0.
\end{eqnarray*}
Therefore we get $c_1^{\m}(M)\geqslant 0.$
\end{example}

We now give an example where $c_1^{I}(M)$ could be negative.

\begin{example}\label{exam2}
Let $(R,\m)$ be zero dimensional Gorenstein local ring with reduction number of $\m=2$, i.e., ${\m}^2\neq 0$ but ${\m}^3=0$. Then we get
\begin{equation*}
{\m}^2 =\socle(R) = \Hom_R(k,R)\simeq k.\tag{i}
\end{equation*}
We now take $M=k$. Consider the exact sequence
\[
0\longrightarrow {\m}^2\longrightarrow R\longrightarrow \frac{R}{{\m}^2} \longrightarrow 0.
\]
This induces the following isomorphism
\begin{equation*}
\Hom_R(k,\frac{R}{{\m}^2})\simeq \Ext_R^1(k,k).\tag{ii}
\end{equation*}
Let $\mu(\m)=h$, the minimal number of generators of $\m$. Then using (i) and (ii), we get
\begin{eqnarray*}
\sum_{n\geqslant 0}\ell\left( \Hom_R(k,\frac{R}{{\m}^{n+1}})\right)t^n &=& \frac{f^{\m}_M(t)}{1-t}\\
\Rightarrow \quad 1 + h t + \frac{t^2}{1-t} &=& \frac{f^{\m}_M(t)}{1-t}.\\
\mbox{So,}\quad f^{\m}_M(t) &=& t^2+(1-t)(1 + h t).
\end{eqnarray*}
Therefore it follows that $c^{\m}_1(M)=1-h$. So $c^{\m}_1(M)<0$ when $h\geqslant 2$.
\end{example}
A specific example of the above kind is the following:
\begin{example}\label{negacoff}
 Let $e$ be any positive integer greater than $3$. Consider
\[
R=\frac{k[[t^e,t^{e+1},\cdots,t^{2e-2}]]}{(t^e)}
\]
It follows from \cite[3.2]{SJ1} that $R$ is $0$-dimensional Gorenstein local ring with maximal ideal
\[
\m=\frac{(t^e,t^{e+1},\cdots,t^{2e-2})}{(t^e)}.
\]
Clearly $h = \mu(\m)=e-2 \geq 2$.
\end{example}

\section{ Example: Dual Hilbert-Samuel function of an Ulrich module}
\noindent In this section we assume that $R$ and its associated graded ring $G_\m(R)$ are Gorenstein. Let $L=\Syz^1_R(M),$ denote the first syzygy module of $M$.  Set $S=R/{(x)}$, $N=M/{xM}$ and $\n = \m /{xR}$.
The goal of this section is to compute the Dual Hilbert-Samuel function of an Ulrich module.
Recall that a MCM module $M$ is said to be Ulrich if $e_0^\m(M) = \mu(M)$

The following theorem gives a sufficient condition which ensures that the dual Hilbert-Samuel function of $M$ behaves "perfectly"
\wrt\ a superficial element.
\begin{theorem}\label{behevsuper1}
 Assume that $x\in \m \setminus \m^2$ is such that $x^\ast$ is $G_\m(R)\oplus G_\m(L)$-regular. Then
 \[
\varepsilon^0_M(\m,n) - \varepsilon^0_M(\m,n-1)= \varepsilon^0_{N}(\n,n) \quad \mbox{for}\;\; n \geqslant 0.
\]
\end{theorem}
\begin{proof} Consider the following exact sequence
\[
 0\rightarrow \m^n \rightarrow R \rightarrow \frac{R}{\m^n}\rightarrow 0.
\]
This induces that
\begin{equation*}
D^1_\m(M)\simeq \bigoplus_{n\geqslant 0} \Ext^2_R(M,\m^n)\simeq\Ext^2_R(M,{\mathcal R}(\m)).\tag{i}
\end{equation*}
 Also the exact sequence
\[
0\longrightarrow L \longrightarrow R^{\mu(M)} \longrightarrow M\longrightarrow 0
\]
yields the following isomorphism
\begin{equation*}
\Ext^1_R(L,{\mathcal R}(\m))\simeq \Ext^2_R(M,{\mathcal R}(\m)).\tag{ii}
\end{equation*}
From (i) and (ii), it follows that
\begin{equation*}
D^1_\m(M) \simeq \Ext^1_R(L,{\mathcal R}(\m)).\tag{iii}
\end{equation*}
Since $G_\m(R)$ is Gorenstein and $G_\m(L)$ and the homogeneous element $x^\ast\in {\mathcal R}(\m)_1$ is $G_\m(R)\oplus G_\m(L)$-regular then by \cite[11.5]{TJP}, the sequence
\[
 0\rightarrow \Hom_R(L,\m^n)\rightarrow \Hom_R(L,\m^{n+1})\longrightarrow \Hom_R(L,\n^{n+1})\longrightarrow 0
\]
is exact for all $n \geqslant 0$. Therefore it follows that the maps
\[
 \Ext_R^1(L,\m^n)\rightarrow \Ext_R^1(L, \m^{n+1})\quad\mbox{ are injective for all}\;\;n\geqslant 0
\]
Using (iii) the maps
\[
\Ext^1_R(M,\frac{R}{\m^n})\xrightarrow{\;\;xt\;\;}\Ext^1_R(M,\frac{R}{\m^{n+1}})
\]
are injective for all $n\geqslant 1$. Therefore the short exact sequence
\[
0\rightarrow \frac{R}{{\m}^n} \xrightarrow{\;\alpha_n^x\;} \frac{R}{{\m}^{n+1}}\rightarrow \frac{S}{\n^{n+1}} \rightarrow 0.
\]
 induces a short exact of the form
\[
0\rightarrow \Hom_R(M,\frac{R}{\m^n})\xrightarrow{\;\;x\;\;}\Hom_R(M,\frac{R}{\m^{n+1}})\rightarrow\Hom_R(N,\frac{S}{\n^{n+1}})\rightarrow 0.
\]
This proves our assertion.
\end{proof}

\s {\bf The case when $M$ is Ulrich module:}
 Let $M$ be Ulrich $R$-module. Then there exists $J=(x_1,x_2,\ldots,x_d)$, a minimal reduction of $\m$ such that $JM={\m}M$. Therefore we get
 \[
  G_{\m}(M)\simeq \frac{M}{\m M}[\bar{x_1},\bar{x_2},\ldots,\bar{x_d}].
 \]
Let $\mu(M)$ denote minimal number of generators of $M$. Then $M/{JM}\simeq k^{\mu(M)}$.

\s \label{zero-Gor} Assume $(S,\n, k)$ is a  $0$-dimensional Gorenstein local ring  with $G_{n}(S)$ Gorenstein. So there exists a positive integer $r$ such that ${\n}^r\neq 0$ but ${\n}^{r+1}=0$. From the exact sequence
\[
 0\longrightarrow {\n}^i\longrightarrow S \longrightarrow \frac{S}{{\n}^i}\longrightarrow 0
\]
one can verify that
\begin{eqnarray*}
      \Hom_S(k,\frac{S}{{\n}^i}) \simeq
                      \left \{
                     \begin{array}{l}
                   \Ext^1_S(k,{\n}^i)\quad\mbox{for}\;i=1,2,\ldots,r\\
                     k\;\quad\quad\quad\quad\quad\mbox{for}\;i>r.
                     \end{array}
                     \right .
\end{eqnarray*}
Also notice that $\Ext^1_S(k,{\n}^i)=k^{\mu_1({\n}^i)}$, where $\mu_1({\n}^i)$ is the first {\it Bass number} of ${\n}^i.$ Now we can compute the dual Hilbert-Samuel series of $k$ with respect to $\n$.
\begin{eqnarray*}
 D(k,t)&=&\mu_1({\n})+\mu_1({\n}^2) t+\ldots+\mu_1({\n}^r)t^{r-1}+\frac{t^{r+1}}{(1-t)}\\
 &=& \frac{(1-t)\cdot[\sum_{i=0}^{r-1}\mu_1({\n}^{i+1})\cdot t^i)]+t^{r+1}}{(1-t)}
\end{eqnarray*}

\begin{proposition}\label{ulrichprop1}
 Let $M$ be an Ulrich $R$-module.  Then we have
 \[
  D(M,t)=\frac{\mu(M)\cdot\left[ (1-t)\cdot[\sum_{i=0}^{r-1}\mu_1({\n}^{i+1})\cdot t^i)]+t^{r+1}\right] }{{(1-t)}^{d+1}}.
 \]
\end{proposition}
\begin{proof}
Since $M$ is Ulrich $R$-module, there exists $J=(x_1,x_2,\ldots,x_d)$, a minimal reduction of $\m$ such that $JM={\m}M$. Set $\bar M=M/JM$. In \cite{TJPW} we prove that $G\m(\Syz^{A}_{1}(M))$ is \CM. In view of Theorem \ref{behevsuper1}, it follows that
 \begin{eqnarray*}
  D(M,t)&=& \frac{D(\bar M,t)}{{(1-t)}^d} = \frac{D(k^{\mu(M)},t)}{{(1-t)}^d}\\
  &=& \mu(M)\cdot \frac{D(k,t)}{{(1-t)}^d}\\
  &=& \frac{\mu(M)\cdot\left[ (1-t)\cdot[\sum_{i=0}^{r-1}\mu_1({\n}^{i+1})\cdot t^i)]+t^{r+1}\right] }{{(1-t)}^{d+1}}
  \end{eqnarray*}
  The last equality follows from 6.3.
\end{proof}

\providecommand{\bysame}{\leavevmode\hbox to3em{\hrulefill}\thinspace}
\providecommand{\MR}{\relax\ifhmode\unskip\space\fi MR }
\providecommand{\MRhref}[2]{%
  \href{http://www.ams.org/mathscinet-getitem?mr=#1}{#2}
}
\providecommand{\href}[2]{#2}

\end{document}